\RequirePackage{fix-cm}
\documentclass[smallcondensed]{svjour3}     % onecolumn (ditto)
\smartqed  % flush right qed marks, e.g. at end of proof
\usepackage{mathptmx}      % use Times fonts if available on your TeX system
\usepackage{amsmath}
\usepackage{latexsym}
\usepackage{graphicx}
\usepackage{xhfill}
\newtheorem{algorithm}{Algorithm}
\begin{document}
\newcommand{\xfill}[2][1ex]{{%
  \dimen0=#2\advance\dimen0 by #1
  \leaders\hrule height \dimen0 depth -#1\hfill%
}}
\title{The convergence rate of a golden ratio algorithm for equilibrium problems
}

\titlerunning{Rate of convergence of a golden ratio algorithm for EPs}        % if too long for running head

\author{Dang Van Hieu
}

\authorrunning{Dang Van Hieu} % if too long for running head
\institute{Dang Van Hieu \at
              Applied Analysis Research Group, Faculty of Mathematics and Statistics, \\
Ton Duc Thang University, Ho Chi Minh City, Vietnam\\
              \email{dangvanhieu@tdtu.edu.vn}           %  \\
            }

\date{Received: date / Accepted: date}

\maketitle

\begin{abstract}
In this paper, we establish the $R$-linear rate of convergence of a golden ratio algorithm for solving an equilibrium problem in a Hilbert space. 
Several experiments are performed to show the numerical behavior of the algorithm and also to compare with others.
\end{abstract}
%%%%%%%%%%%%%%%%%%%%%%%%%%%%%%%%%%
\section{Introduction}\label{intro}
%%%%%%%%%%%%%%%%%%%%%%%%%%%%%%%%%%%%%
The equilibrium problem (EP), in the sense of Blum and Oettli \cite{BO1994,MO1992}, is a general model which unifies in a simple form various 
mathematical models as optimization problems, variational inequalities, fixed point problems and Nash equilibrium problem, see in \cite{FP2002,K2007}. 
Together with studying the existence of solution, the iterative method for approximating solutions of problem (EP) is also studied and has received a lot 
of interest by many authors. Some notable iterative methods for solving problem (EP) can be found, for example, in \cite{AH2018,FA1997,HCX2018,H2018MMOR,Hieu2018NUMA,K2003,M1999,QMH2008,SS2011,SNN2013} and the references 
cited therein. In this paper, we concern with the rate of convergence of an algorithm, namely the golden ratio algorithm, for solving problem (EP) in 
a Hilbert space.\\[.1in]
%%%%%%%%%%%%%%%%%%%%%%%%%%%%%%%%%%
One of the most popular methods for approximating a solution of problem (EP) is the proximal point method \cite{M1999,K2003} which is constructed 
around the resolvent of cost bifunction. The resolvent of a bifunction is regarding a regularization equilibirum subproblem. The regularization solutions 
can converge finitely or asymptotically to a solution of original problem (EP). However, the problem of this method is that in practice the computation 
of a value of resolvent mapping at a point, i.e., finding a regularization solution, is often expensive and we have to use a inner loop. This is of course time-consuming. 
Another method for solving problem (EP) is the proximal-like method \cite{FA1997} which incluses optimization problems. Recently, the authors in \cite{QMH2008} have 
investigated and extended further the convergence of the proximal-like method in \cite{FA1997} under other assumptions that the cost bifunction 
is pseudomonotone and satisfies a Lipschitz-type condition. The proximal-like methods in \cite{FA1997,QMH2008} are also called the extragradient 
method due to the early results of Korpelevich \cite{K1976} on saddle point problem. We can use some known convex optimization programming 
to compute in the extragradient method. This is reason to explain why the latter is often easier to solve numerically than the proximal point method in 
\cite{K2003,M1999}. \\[.1in]
%%%%%%%%%%%%%%%%%%%%%%%%%%%%%%%%%%
In view of the extragradient method in \cite{FA1997,QMH2008}, it is seen that per each iteration we have to solve two convex optimization subproblems 
on feasible set. This can be expansive if cost bifunction and feasible set have complex structures. Then, in recent years, many variants of the extragradient 
method have been developed to reduce the complexity of algorithm as well as to weaken assumptions imposed on cost bifunction, see, for example, in \cite{H2017NUMA,H2017JIMO,LS2016}. \\[.1in]
%%%%%%%%%%%%%%%%%%%%%%%%%%%%%%%%%%
Very recently, motivated by the results of Malitsky in \cite{M2018}, Vinh in \cite{V2018} has presented and analyzed the weak convergence of a new algorithm, 
namely Golden Ratio Algorithm (GRA), under the two hypotheses of pseudomonotonicity and Lipschitz-type condition imposed on cost bifunction. The main advantage 
of GRA is that over each iteration the algorithm only requires to compute a convex optimization program. A modification of GRA is also presented where the 
algorithm uses a sequence of variable stepsizes being diminishing and non-summable. By strengthening the pseudomonotonicity by the strongly pseudomonotonicity, 
Vinh has established the strong convergence of this modification. A question arises naturally here as follows:\\[.1in]
%%%%%%%%%%%%%%%%%%%%%%%%%%%%%%%%%%
\textbf{Question:} \textit{Can we establish the convergence rate of GRA?}\\[.1in]
%%%%%%%%%%%%%%%%%%%%%%%%%%%%%%%%%%
In this paper, we give a positive answer to the aforementioned question. Under the assumptions of strongly pseudomonotonicity and Lipschitz-type condition 
of cost bifunction, we prove that algorithm GRA converges $R$-linearly. The remainder of this paper is organized as follows: In Sect. \ref{pre} we collect 
some definitions and preliminary results used in the paper. Sect. \ref{main} deals with presenting algorithm GRA and establishing its convergence rate. Finally, 
in Sect. \ref{example}, several numerical results are reported to show the behavior of the algorithm in comparison with others.
%%%%%%%%%%%%%%%%%%%%%%%%%%%%%%%%%%%%%%%%
\section{Preliminaries}\label{pre}
%%%%%%%%%%%%%%%%%%%%%%%%%%%%%%%%%%%%%%%%
Let $H$ be a real Hilbert space and $C$ be a nonempty closed convex subset of $H$. Let $f:C\times C\to \Re$ be a bifunction with $f(x,x)=0$ for all $x\in C$. 
The equilibrium problem for the bifunction $f$ on $C$ is stated as follows:
$$
\mbox{Find}~x^*\in C~\mbox{such that}~f(x^*,y)\ge 0,~\forall y\in C.
\eqno{\rm (EP)}
$$
Let $g:C\to \Re$ be a proper, lower semicontinuous, convex function. The proximal operator ${\rm prox}_{\lambda g}$ of $g$ with some $\lambda>0$ 
is defined by 
$$
{\rm prog}_{\lambda g}(z)=\arg\min \left\{\lambda g(x)+\frac{1}{2}||x-z||^2:x\in C\right\},~z\in H.
$$
The following is an important property of the proximal mapping (see \cite{BC2011} for more details)
%%%%%%%%%%%%%%%%%%%%%%%%%%%%%%%%%%%%%%%%%%%
\begin{lemma}\label{prox}
$
\bar{x}={\rm prog}_{\lambda g}(z) \Leftrightarrow \left\langle \bar{x}-z,x-\bar{x}\right\rangle \ge \lambda \left(g(\bar{x})-g(x)\right),~\forall x\in C.
$
\end{lemma}
\begin{remark}\label{rem1}
From Lemma \ref{prox}, it is easy to show that if $x={\rm prox}_{\lambda g}(x)$ then 
$$x\in {\rm arg}\min\left\{g(y):y\in C\right\}:=\left\{x\in C: g(x)=\min_{y\in C}g(y)\right\}.$$
\end{remark}
In any Hilbert space, we have the following result, see, e.g., in \cite[Corollary 2.14]{BC2011}.
\begin{lemma}\label{eq} For all $x,~y\in H$ and $\alpha\in \Re$, the following equality always holds
$$
||\alpha x+(1-\alpha)y||^2=\alpha ||x||^2+(1-\alpha)||y||^2-\alpha(1-\alpha)||x-y||^2.
$$
\end{lemma}
Throughout this paper, we assume that the function $f(\cdot,y)$ is convex, lower semicontinuous and $f(x,\cdot)$ is hemicontinuous on $C$ for each 
$x,~y\in C$. In order to establish the rate of convergence of algorithm GRA, we consider the following assumptions:\\[.1in]
%%%%%%%%%%%%%%%%%%%%%%%%%%%%%%%%%%
(SP):\, $f$ is strongly pseudomonotone, i.e., there exists a constant $\gamma>0$ such that
$$
f(x,y)\ge 0 \Longrightarrow f(y,x) \le-\gamma ||x-y||^2\,~\mbox{for all}~ x,y\in C.
$$
(LC):\, $f$ satisfies the Lipschitz-type condition, i.e., there exist $c_1>0,c_2>0$ such that
$$
 f(x,y) + f(y,z) \geq f(x,z) - c_1||x-y||^2 - c_2||y-z||^2\, ~ \mbox{for all}~ x,y,z \in C.
 $$
Under these assumptions, problem (EP) has the unique solution, denoted by $x^\dagger$. For solving problem (EP) with the conditions (SP) 
and (LC), Vinh in \cite{V2018} introduced the following modified golden ratio algorithm (MGRA1): Set $\varphi=\frac{1+\sqrt{5}}{2}$ 
(the golden ratio), take $x_0,~y_1\in C$ and compute 
 $$
\left \{
\begin{array}{ll}
x_n=\frac{(\varphi-1)y_n+x_{n-1}}{\varphi},\\
y_{n+1}=\arg\min \left\{\lambda_n f(y_n,y)+\frac{1}{2}||x_n-y||^2:y\in C\right\},
\end{array}
\right.
\eqno{\rm (MGRA1)}
$$
where $\left\{\lambda_n\right\}\subset (0,+\infty)$ is a sequence satisfying the hypotheses:
$$
{\rm (C1):\, \lim_{n\to\infty}\lambda_n=0},\quad {\rm (C2):\,\sum_{n=0}^\infty \lambda_n=+\infty}.
$$
Without the condition (LC), Vinh introduced the second modified golden ratio algorithm (MGRA2) which is based on the early results in \cite{SS2011}: 
Take $x_0,~y_1\in C$ and compute 
$$
\left \{
\begin{array}{ll}
x_n=\frac{(\varphi-1)y_n+x_{n-1}}{\varphi},\\
g_n\in \partial f(y_n,.)(y_n),~\lambda_n=\frac{\beta_n}{\max \left\{1,||g_n||\right\}},\\
y_{n+1}=P_C(x_n-\lambda_n g_n),
\end{array}
\right.
\eqno{\rm (MGRA2)}
$$
where $\left\{\beta_n\right\}\subset (0,+\infty)$ is a sequence satisfying the hypotheses:
$$
{\rm (C3):\, \sum_{n=0}^\infty \beta_n=+\infty},\quad {\rm (C4):\,\sum_{n=0}^\infty \beta_n^2<+\infty}.
$$
In \cite{V2018}, Vinh proved that the sequence $\left\{x_n\right\}$ generated by MGRA1 or MGRA2 converges strongly 
to the unique solution $x^\dagger$ of problem (EP) without any estimate of the rate of convergence of the sequence 
$\left\{x_n\right\}$.  We remark here that algorithm MGRA1 cannot converge linearly. Indeed, consider our problem 
for $C=H=\Re$ and $f(x,y)=x(y-x)$. The unique solution of the problem is $x^\dagger=0$. It is easy to see that $f$ 
satisfies the conditions (SP) and (LC). It follows from the definition of $y_{n+1}$ and a simple computation that 
\begin{equation}\label{t1}
y_{n+1}=x_n-\lambda_n y_n.
\end{equation}
Since $x_n=\frac{(\varphi-1)y_n+x_{n-1}}{\varphi}$, we obtain that $y_n=\frac{\varphi}{\varphi-1}x_n-\frac{1}{\varphi-1}x_{n-1}$. 
Thus, from the relation (\ref{t1}), we obtain
$$
\frac{\varphi}{\varphi-1}x_{n+1}-\frac{1}{\varphi-1}x_{n}=x_n-\lambda_n \left[ \frac{\varphi}{\varphi-1}x_n-\frac{1}{\varphi-1}x_{n-1}\right],
$$
or
\begin{equation}\label{t2}
x_{n+1}=(1-\lambda_n)x_n+\frac{\lambda_n}{\varphi} x_{n-1}.
\end{equation}
Since $\lambda_n\to 0$, without loss of generality, we can assume that $\left\{\lambda_n\right\}\subset (0,1)$. If choose $x_0,~y_1>0$, 
from the definition of $x_1$, we obtain that $x_1>0$. Thus, from the relation (\ref{t2}), we get by the induction that $x_n>0$ for all $n\ge 0$.
Now, assume that $\left\{x_n\right\}$ converges linearly to the solution $x^\dagger=0$, i.e., there exists a number $\sigma\in (0,1)$ such that 
$||x_{n+1}-x^\dagger||\le \sigma ||x_{n}-x^\dagger||$ or $x_{n+1}\le \sigma x_n$ for all $n\ge 0$. This together with the relation (\ref{t2}) implies 
that 
$$
x_{n+1}=(1-\lambda_n)x_n+\frac{\lambda_n}{\varphi} x_{n-1}\ge (1-\lambda_n)\frac{x_{n+1}}{\sigma}+\frac{\lambda_n}{\varphi} \frac{x_{n}}{\sigma}
\ge (1-\lambda_n)\frac{x_{n+1}}{\sigma}+\frac{\lambda_n}{\varphi} \frac{x_{n+1}}{\sigma^2}.
$$
Thus
\begin{equation}\label{t3}
1\ge \frac{1-\lambda_n}{\sigma}+\frac{\lambda_n}{\varphi\sigma^2},~\forall n\ge 0.
\end{equation}
Passing to the limit in the relation (\ref{t3}) as $n\to\infty$ and using the hypothesis (C1), we obtain that $1\ge \frac{1}{\sigma}$. This is contrary because 
$\sigma \in (0,1)$. This says that the sequence $\left\{x_n\right\}$ does not converge linearly to the solution $x^\dagger=0$. Thus, algorithm MGRA1 cannot 
converge linearly.\\[.1in]
%%%%%%%%%%%%%%%%%%%%%%%%%%%%%%%%%
In the next section, we will present in details algorithm GRA with a fixed stepsize and establish the $R$-linear rate of convergence of the algorithm.
%%%%%%%%%%%%%%%%%%%%%%%%%%%%%%%%%
\section{The $R$-linear rate of convergence of GRA}\label{main}
%%%%%%%%%%%%%%%%%%%%%%%%%%%%%%%%%
In this section, we study the rate of convergence of the following golden ratio algorithm.\\
%%%%%%%%%%%%%%%%%%%%%%%%%%%%%%%%%%%%%
\noindent\rule{12.1cm}{0.4pt} 
\begin{algorithm}[Golden Ratio Algorithm for Equilibrium Problem].\label{alg1}\\
\noindent\rule{12.1cm}{0.4pt}\\
\textbf{Initialization:} Set $\varphi=\frac{\sqrt{5}+1}{2}$. Choose $\bar{x}_0\in H,~x_1\in C$ and $\lambda\in \Re$ such that 
$0<\lambda < \frac{\varphi}{4\max \left\{c_1,c_2\right\}}$\\
\noindent\rule{12.1cm}{0.4pt}\\
\textbf{Iterative Steps:} Assume that $\bar{x}_{n-1}\in H, ~x_n\in C$ are known, calculate $x_{n+1}$ as follows:
$$
\left \{
\begin{array}{ll}
\bar{x}_n=\frac{(\varphi-1)x_n+\bar{x}_{n-1}}{\varphi},\\
x_{n+1}=\mbox{\rm prox}_{\lambda f(x_n,.)}(\bar{x}_n).
\end{array}
\right.
$$
\end{algorithm}
\noindent\rule{12.1cm}{0.4pt}\\[.1in]
%%%%%%%%%%%%%%%%%%%%%%%%%%%%%%%%
We can use the following stopping criterion for Algorithm \ref{alg1}: If $x_{n+1}=x_n=\bar{x}_n$ then stop and $x_n$ is the solution of problem (EP).
This follows from the definition of $x_{n+1}$ and Remark \ref{rem1}. Thus, if Algorithm \ref{alg1} terminates then the solution of the problem can be 
found. Otherwise, we have the following main result.
%%%%%%%%%%%%%%%%%%%%%%%%%%%%%%%%%%%%%%
\begin{theorem}\label{theo3}
Under the conditions (SM) and (LC), the sequence $\left\{x_n\right\}$ generated by Algorithm \ref{alg1} converges $R$-linearly to the unique solution 
$x^\dagger$ of problem (EP).
\end{theorem}
\begin{proof}
It follows from the definition of $x_{n+1}$ and Lemma \ref{prox} that 
\begin{equation}\label{eq:1}
\left\langle \bar{x}_n-x_{n+1}, x-x_{n+1}\right\rangle \le \lambda \left( f(x_n,x)-f(x_n,x_{n+1})\right), ~\forall x\in C,
\end{equation}
which, with $x=x^\dagger$, follows that 
\begin{equation}\label{eq:2}
2\left\langle \bar{x}_n-x_{n+1}, x^\dagger-x_{n+1}\right\rangle \le 2\lambda \left( f(x_n,x^\dagger)-f(x_n,x_{n+1})\right).
\end{equation}
Applying the equality $2 \left\langle a,b\right\rangle=||a||^2+||b||^2-||a-b||^2$ to the relation (\ref{eq:2}), we obtain
\begin{eqnarray}
||\bar{x}_n-x_{n+1}||^2+||x_{n+1}-x^\dagger||^2-||\bar{x}_n-x^\dagger||^2\le 2\lambda \left( f(x_n,x^\dagger)-f(x_n,x_{n+1})\right). \label{eq:3}
\end{eqnarray}
Using the relation (\ref{eq:1}) with $n:=n-1$, we obtain
\begin{equation}\label{eq:4}
\left\langle \bar{x}_{n-1}-x_n, x-x_n\right\rangle \le \lambda \left( f(x_{n-1},x)-f(x_{n-1},x_n)\right), ~\forall x\in C.
\end{equation}
Substituting $x=x_{n+1}$ into (\ref{eq:4}), we get 
\begin{equation}\label{eq:5}
\left\langle \bar{x}_{n-1}-x_n, x_{n+1}-x_n\right\rangle \le \lambda \left( f(x_{n-1},x_{n+1})-f(x_{n-1},x_n)\right).
\end{equation}
Noting that from the definition of $\bar{x}_n$, we have $\bar{x}_{n-1}-x_n=\varphi (\bar{x}_n-x_n)$. Then, from the relation (\ref{eq:5}), 
we come the following estimate,
%\begin{equation}\label{eq:6}
$$
2\varphi\left\langle \bar{x}_n-x_n, x_{n+1}-x_n\right\rangle \le 2\lambda \left( f(x_{n-1},x_{n+1})-f(x_{n-1},x_n)\right).
$$
%\end{equation}
Thus, as the relation (\ref{eq:3}), developing the product $2\left\langle \bar{x}_n-x_n, x_{n+1}-x_n\right\rangle$, we get 
\begin{equation}\label{eq:7}
\varphi ||\bar{x}_n-x_n||^2+\varphi ||x_{n+1}-x_n||^2-\varphi ||x_{n+1}-\bar{x}_n||^2 \le 2\lambda \left( f(x_{n-1},x_{n+1})-f(x_{n-1},x_n)\right).
\end{equation}
Adding the relations (\ref{eq:3}) and (\ref{eq:7}), and using the condition (LC), we obtain
\begin{eqnarray}
&&||x_{n+1}-x^\dagger||^2-||\bar{x}_n-x^\dagger||^2+(1-\varphi)||x_{n+1}-\bar{x}_n||^2+\varphi ||\bar{x}_n-x_n||^2\nonumber\\ 
&&+\varphi ||x_{n+1}-x_n||^2\nonumber\\ 
&\le& 2\lambda f(x_n,x^\dagger)- 2\lambda \left[f(x_{n-1},x_n) + f(x_n,x_{n+1}) - f(x_{n-1},x_{n+1})\right]\nonumber\\ 
&\le& 2\lambda f(x_n,x^\dagger)+ 2\lambda c_1 ||x_{n-1}-x_n||^2 + 2\lambda c_2 ||x_n - x_{n+1}||^2.\label{eq:8}
\end{eqnarray}
Since $0<\lambda < \frac{\varphi}{4\max\left\{c_1,c_2\right\}}$, we reduce that 
$\lambda<\frac{\varphi}{4c_1}$ and $\lambda<\frac{\varphi}{4c_1}$. Thus, there exists a number $\epsilon\in (0,1)$ such that 
$$
\lambda<\frac{\epsilon\varphi}{4c_1}\qquad \mbox{and} \qquad \lambda<\frac{\varphi}{4c_2}.
$$
Combining these inequalities with the relation (\ref{eq:8}), we get 
\begin{eqnarray*}
||x_{n+1}-x^\dagger||^2&-&||\bar{x}_n-x^\dagger||^2+(1-\varphi)||x_{n+1}-\bar{x}_n||^2+\varphi ||\bar{x}_n-x_n||^2\\
&&+\varphi ||x_{n+1}-x_n||^2\\ 
&\le& 2\lambda f(x_n,x^\dagger)+ \frac{\epsilon\varphi}{2} ||x_{n-1}-x_n||^2 + \frac{\varphi}{2} ||x_n - x_{n+1}||^2.
\end{eqnarray*}
Since $x^\dagger$ is the solution of problem (EP) and $x_n\in C$, $f(x_n,x^\dagger)\ge 0$. Hence, from the strong pseudomonotonicity of $f$, we derive 
$f(x_n,x^\dagger)\le -\gamma ||x_n-x^\dagger||^2$. This together with the last inequality implies that
\begin{eqnarray}
&&||x_{n+1}-x^\dagger||^2+\frac{\varphi}{2} ||x_n - x_{n+1}||^2\le ||\bar{x}_n-x^\dagger||^2+ \frac{\epsilon\varphi}{2} ||x_{n-1}-x_n||^2\nonumber\\ 
&-&(1-\varphi)||x_{n+1}-\bar{x}_n||^2-\varphi ||\bar{x}_n-x_n||^2 -2\lambda \gamma ||x_n-x^\dagger||^2.\label{eq:9}
\end{eqnarray}
Moreover, from the definition of $\bar{x}_n$ and Lemma \ref{eq}, we have
\begin{eqnarray}\label{eq:10}
||x_{n+1}-x^\dagger||^2&=&\frac{\varphi}{\varphi-1}||\bar{x}_{n+1}-x^\dagger||^2-\frac{1}{\varphi-1}||\bar{x}_n-x^\dagger||^2\nonumber\\ 
&&+\frac{\varphi}{(\varphi-1)^2}||\bar{x}_{n+1}-\bar{x}_n||^2\nonumber\\
&=&\frac{\varphi}{\varphi-1}||\bar{x}_{n+1}-x^\dagger||^2-\frac{1}{\varphi-1}||\bar{x}_n-x^\dagger||^2\nonumber\\ 
&&+\frac{1}{\varphi}||x_{n+1}-\bar{x}_n||^2.
\end{eqnarray}
Thus, from the relations (\ref{eq:9}), (\ref{eq:10}) and the fact $1-\varphi-\frac{1}{\varphi}=0$, we obtain
\begin{eqnarray}
&&\frac{\varphi}{\varphi-1}||\bar{x}_{n+1}-x^\dagger||^2+\frac{\varphi}{2} ||x_n - x_{n+1}||^2\le \frac{\varphi}{\varphi-1} ||\bar{x}_n-x^\dagger||^2+ \frac{\epsilon\varphi}{2} ||x_{n-1}-x_n||^2\nonumber\\ 
&&-(1-\varphi-\frac{1}{\varphi})||x_{n+1}-\bar{x}_n||^2-\varphi ||\bar{x}_n-x_n||^2 -2\lambda \gamma ||x_n-x^\dagger||^2\nonumber\\
&\le& \frac{\varphi}{\varphi-1} ||\bar{x}_n-x^\dagger||^2+ \frac{\epsilon\varphi}{2} ||x_{n-1}-x_n||^2 -2\lambda \gamma ||x_n-x^\dagger||^2.\label{eq:11}
\end{eqnarray}
Note that from the definition of $\bar{x}_n$, we obtain
$
x_n=\frac{\varphi}{\varphi-1}\bar{x}_n-\frac{1}{\varphi-1}\bar{x}_{n-1}.
$
Thus, it follows from Lemma \ref{eq} that
\begin{eqnarray}\label{eq:12}
||x_n-x^\dagger||^2&=&\frac{\varphi}{\varphi-1}||\bar{x}_n-x^\dagger||^2-\frac{1}{\varphi-1}||\bar{x}_{n-1}-x^\dagger||^2+\frac{\varphi}{(\varphi-1)^2}||\bar{x}_n-\bar{x}_{n-1}||^2\nonumber\\
&\ge&\frac{\varphi}{\varphi-1}||\bar{x}_n-x^\dagger||^2-\frac{1}{\varphi-1}||\bar{x}_{n-1}-x^\dagger||^2.
\end{eqnarray}
From the relations (\ref{eq:11}) and (\ref{eq:12}), we see that
\begin{eqnarray}
\frac{\varphi}{\varphi-1}||\bar{x}_{n+1}-x^\dagger||^2&+&\frac{\varphi}{2}||x_{n+1}-x_n||^2\le\frac{\varphi}{\varphi-1}(1-2\gamma\lambda)||\bar{x}_n-x^\dagger||^2\nonumber\\
&&+\frac{2\gamma\lambda}{\varphi-1}||\bar{x}_{n-1}-x^\dagger||^2+\frac{\epsilon\varphi}{2} ||x_n-x_{n-1}||^2.\label{eq:13}
\end{eqnarray}
Set $a_n=\frac{\varphi}{\varphi-1}||\bar{x}_n-x^\dagger||^2$, $b_n=\frac{\varphi}{2}||x_{n}-x_{n-1}||^2$ and $\alpha=2\lambda \gamma$, the inequality (\ref{eq:13}) can be rewritten as
\begin{eqnarray}
a_{n+1}+b_{n+1}&\le&(1-\alpha)a_n+\frac{\alpha}{\varphi}a_{n-1}+\epsilon b_n.\label{eq:14}
\end{eqnarray}
Let $r_1>0$ and $r_2>0$. Now, the relation (\ref{eq:14}) can be rewritten as 
\begin{eqnarray}
a_{n+1}+r_1 a_n+b_{n+1}&\le& r_2(a_n+r_1 a_{n-1})+\epsilon b_n\nonumber\\
&&+(1-\alpha-r_2+r_1)a_n+(\frac{\alpha}{\varphi} -r_1 r_2) a_{n-1}.\label{h39}
\end{eqnarray}
Choose $r_1>0$ and $r_2>0$ such that $1-\alpha-r_2+r_1=0$ and $\frac{\alpha}{\varphi} -r_1 r_2=0$. Thus
%\begin{equation}\label{h40}
$$
r_1=\frac{\alpha-1+\sqrt{(\alpha-1)^2+\frac{4\alpha}{\varphi}}}{2}\quad \mbox{and}\quad r_2=\frac{1-\alpha+\sqrt{(\alpha-1)^2+\frac{4\alpha}{\varphi}}}{2}.
$$
%\end{equation}
Consider the function 
$$
f(t)=\frac{1-t+\sqrt{(t-1)^2+\frac{4t}{\varphi}}}{2}, ~t\in [0,+\infty).
$$
We have 
$$
f'(t)=\frac{-1+\frac{t-1+\frac{2}{\varphi}}{\sqrt{(t-1)^2+\frac{4t}{\varphi}}}}{2}=\frac{\frac{4}{\varphi}(\frac{1}{\varphi}-1)}{2\sqrt{(t-1)^2+\frac{4t}{\varphi}}\left(t-1+\frac{2}{\varphi}+\sqrt{(t-1)^2+\frac{4t}{\varphi}}\right)}<0
$$
because of $\frac{1}{\varphi}<1$. Thus, $f(t)$ is non-increasing on $[0,+\infty)$. Hence, $0<r_2=f(\alpha)<f(0)=1$. Now, set $\theta=\max \left\{\epsilon,r_2\right\}$ 
and note that $\theta\in (0,1)$ then from the relation (\ref{h39}), we obtain
\begin{eqnarray}
a_{n+1}+r_1 a_n+b_{n+1}&\le& \theta (a_n+r_1 a_{n-1}+ b_n),~\forall n\ge 1.\label{h41}
\end{eqnarray}
Thus, we obtain by the induction that 
\begin{eqnarray}
a_{n+1}+r_1 a_n+b_{n+1}&\le& \theta^{n} (a_{1}+r_1 a_{0}+ b_{1}).\label{h41}
\end{eqnarray}
We can reduce that 
$$
\frac{\varphi}{\varphi-1}||\bar{x}_{n+1}-x^\dagger||^2=a_{n+1}\le \theta^{n} (a_{1}+r_1 a_{0}+ b_{1}),
$$
or
$$
||\bar{x}_{n+1}-x^\dagger||^2\le M\theta^n,
$$
where $M=\frac{(\varphi-1)(a_{1}+r_1 a_{0}+ b_{1})}{\varphi}$, or the sequence $\left\{\bar{x}_n\right\}$ converges $R$-linearly. Since 
$x_n=\frac{\varphi}{\varphi-1}\bar{x}_n-\frac{1}{\varphi-1}\bar{x}_{n-1}=\left(1+\frac{1}{\varphi-1}\right)\bar{x}_n-\frac{1}{\varphi-1}\bar{x}_{n-1}$, 
we derive
$$ 
||x_n-x^\dagger||\le \left(1+\frac{1}{\varphi-1}\right)||\bar{x}_n-x^\dagger||+\frac{1}{\varphi-1}||\bar{x}_{n-1}-x^\dagger||.
 $$
Hence, the sequence $\left\{x_n\right\}$ also converges $R$-linearly. This completes the proof.
\end{proof}
%%%%%%%%%%%%%%%%%%%%%%%%%%%%%%%%%%%%%%%%%%%%%%%%%%%%%%%%
\section{Numerical experiments}\label{example}
%%%%%%%%%%%%%%%%%%%%%%%%%%%%%%%%%%%
In this section, we present some experiments to illustrate the numerical behavior of Algorithm \ref{alg1} (GRA) and also to compare with other algorithms. 
Two algorithms used here to compare are MGRA1 \cite[Algorithm 4.1]{V2018} and MGRA2 \cite[Algorithm 4.2]{V2018}. The convergence of these algorithms 
are established under a same assumption of strong pseudomonotonicity of bifunction, while the Lipschitz-type condition are only for Algorithm \ref{alg1} 
and MGRA1. In order to show the computational performance of the algorithms, we describe the behavior of the sequence 
$D_n=||x_n-{\rm prox}_{\lambda f(x_n,.)}(x_n)||^2,~n=0,~1,2,\ldots$ when number of iterations (\# Iterations) is performed or execution time 
(Elapsed Time) in the second elapses. Noting that $D_n=0$ iff $x_n$ is the solution of the considered problem.\\[.1in]
%%%%%%%%%%%%%%%%%%%%%%%%%%%%%%%%%%%%%%%%%%%%%%%%%%%%%%
Considering a generalization of the Nash-Cournot oligopolistic equilibrium model in \cite{CKK2004,FP2002} with the affine price and fee-fax functions. 
Assume that there are $m$ companies that produce a commodity. Let $x$ denote the vector whose entry $x_j$ stands for the quantity  of
 the commodity produced  by company $j$. We suppose that the price $p_j(s)$ is a decreasing affine function of $s$ with $s= \sum_{j=1}^m x_j$, 
i.e., $p_j(s)= \alpha_j - \beta_j s$, where $\alpha_j > 0$, $\beta_j > 0$. Then the profit made by company $j$ is given by $f_j(x)= p_j(s) x_j -
c_j( x_j)$, where $c_j(x_j)$ is the tax and fee for generating $x_j$. Suppose that $C_j=[x_j^{\min},x_j^{\max}]$ is the strategy set of company $j$, then the
strategy set of the model is $C:= C_1\times C_2 \times ...\times C_m$. Actually, each company seeks to  maximize its profit by choosing the
corresponding production level under the presumption that the production of the other companies is a parametric input.
 A commonly used approach to this model is based upon the famous Nash equilibrium concept. We recall  that a point $x^* \in C=C_1\times C_2
\times\cdots\times C_m$ is  an equilibrium point of the model   if
    $$f_j(x^*) \geq f_j(x^*[x_j]) \ \forall x_j \in C_j, \ \forall  j=1,2,\ldots,m,$$
    where the vector $x^*[x_j]$ stands for the vector obtained from
    $x^*$ by replacing $x^*_j$ with $x_j$.
By taking $f(x, y):= \psi(x,y)-\psi(x,x)$ with $\psi(x,y):=  -\sum_{j = 1}^m f_j(x[y_j])$, 
the problem of finding a Nash equilibrium point of the model can be formulated as:
$$\mbox{Find}~x^* \in C~\mbox{such that}~ f(x^*,x) \geq 0 ~ \forall x \in C. $$
Now, assume that the tax-fee function $c_j(x_j)$ is increasing and affine for every $j$. This assumption means that both of the tax and fee  
for producing a unit are increasing as the quantity of the production gets larger. In that case, the bifunction $f$ can be formulated in 
the form 
$$f(x,y)=\left\langle Px+Qy+q,y-x\right\rangle,$$
where $q\in \Re^m$ and $P,~Q$ are two matrices of order $m$ such that $Q$ is symmetric positive semidefinite and $Q-P$ is symmetric 
negative semidefinite. The bifucntion $f$ satisfies the Lipschitz-type condition with $c_1=c_2=||P-Q||/2$. In order to assure that all the algorithms 
can work, the datas here are generated randomly such that $Q-P$ is symmetric negative \textit{definite}. In this case, the bifunction $f$ is 
strongly pseudomonotone. We perform the numerical computations in $\Re^m$ with $m=100,~200,~300$; the feasible set is $C=[-2,5]^m$; 
the starting point $x_1=\bar{x}_0$ are generated randomly in the interval $[0,1]$. The datas are generated as follows: All the entries of $q$ is generated randomly 
and uniformly in $(-2,2)$ and the two matrices $P,~Q$ are also 
generated randomly\footnote{We randomly choose $\lambda_{1k}\in (-2,0),~\lambda_{2k}\in (0,2),~ k=1,\ldots,m$. We set $\widehat{Q}_1$, $\widehat{Q}_2$ 
as two diagonal matrixes with eigenvalues $\left\{\lambda_{1k}\right\}_{k=1}^m$ and $\left\{\lambda_{2k}\right\}_{k=1}^m$, respectively. Then, we 
construct a positive semidefinite matrix $Q$ and a negative definite matrix $T$ by using random orthogonal matrixes with $\widehat{Q}_2$ 
and $\widehat{Q}_1$, respectively. Finally, we set $P=Q-T$} such that their conditions hold. All the optimization subproblems are effectively solved 
by the function \textit{quadprog} in Matlab 7.0.\\[.1in]
%%%%%%%%%%%%%%%%%%%%%%%%%%%%%%%%%%%
We take $\lambda=\lambda_i:=\frac{p_i\varphi}{4c_1}$, $i=1,2,3,4$ for Algorithm \ref{alg1} (GRA), where $p_1=0.9,~p_2=0.7,~p_3=0.5,~p_4=0.3$ and 
$\lambda_n=\beta_n=\frac{1}{n+1}$ for the MGRA1 and MGRA2. The numerical results are shown in Figs. \ref{fig1} - \ref{fig6}. In view of these figures, 
we see that Algorithm \ref{alg1} works better than other algorithms in both execution time and number of iterations.

\begin{figure}[!ht]
\begin{minipage}[b]{0.45\textwidth}
\centering
\includegraphics[height=5cm,width=6cm]{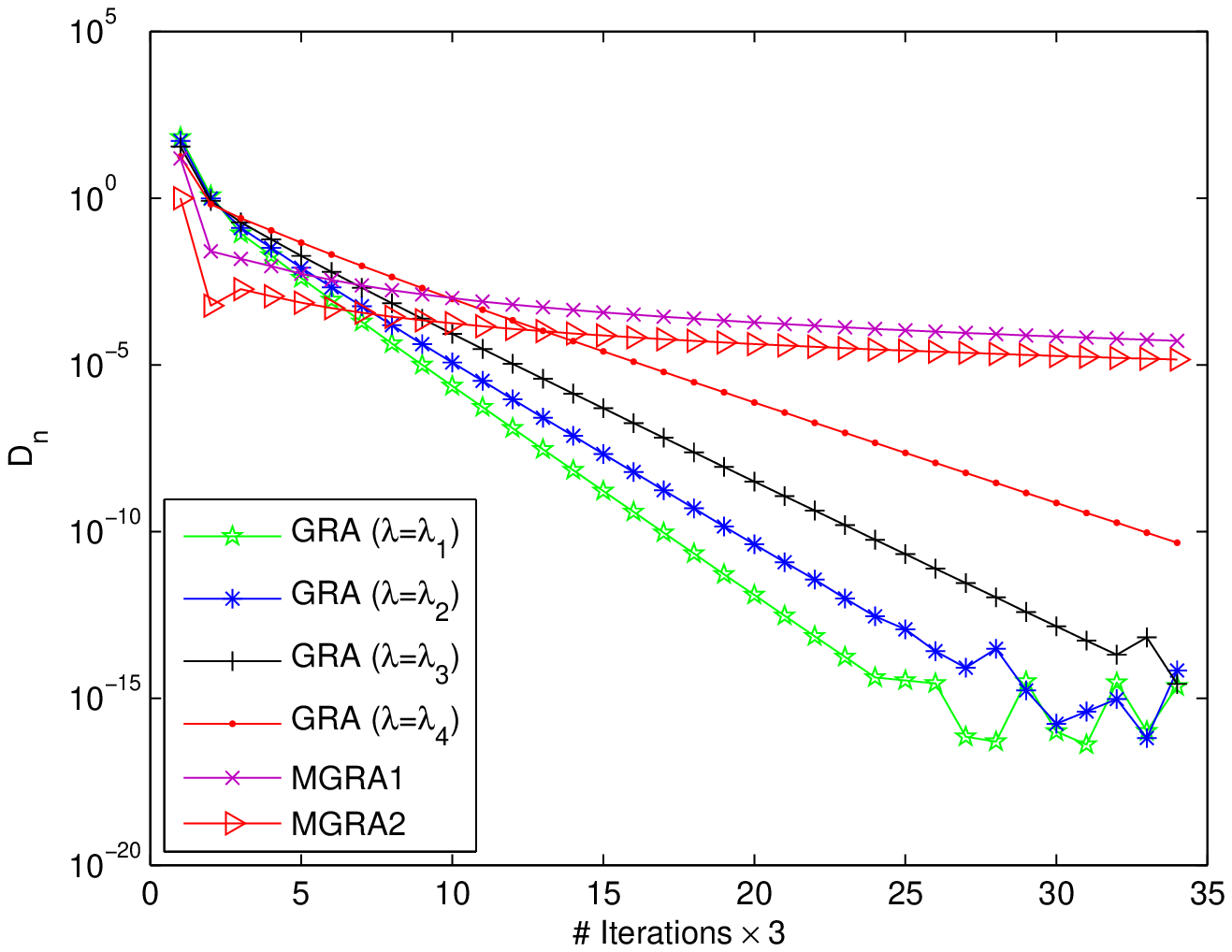}
\caption{$D_n$ and \# Iterations (in $\Re^{100}$)}\label{fig1}
\end{minipage}
\hfill
\begin{minipage}[b]{0.45\textwidth}
\centering
\includegraphics[height=5cm,width=6cm]{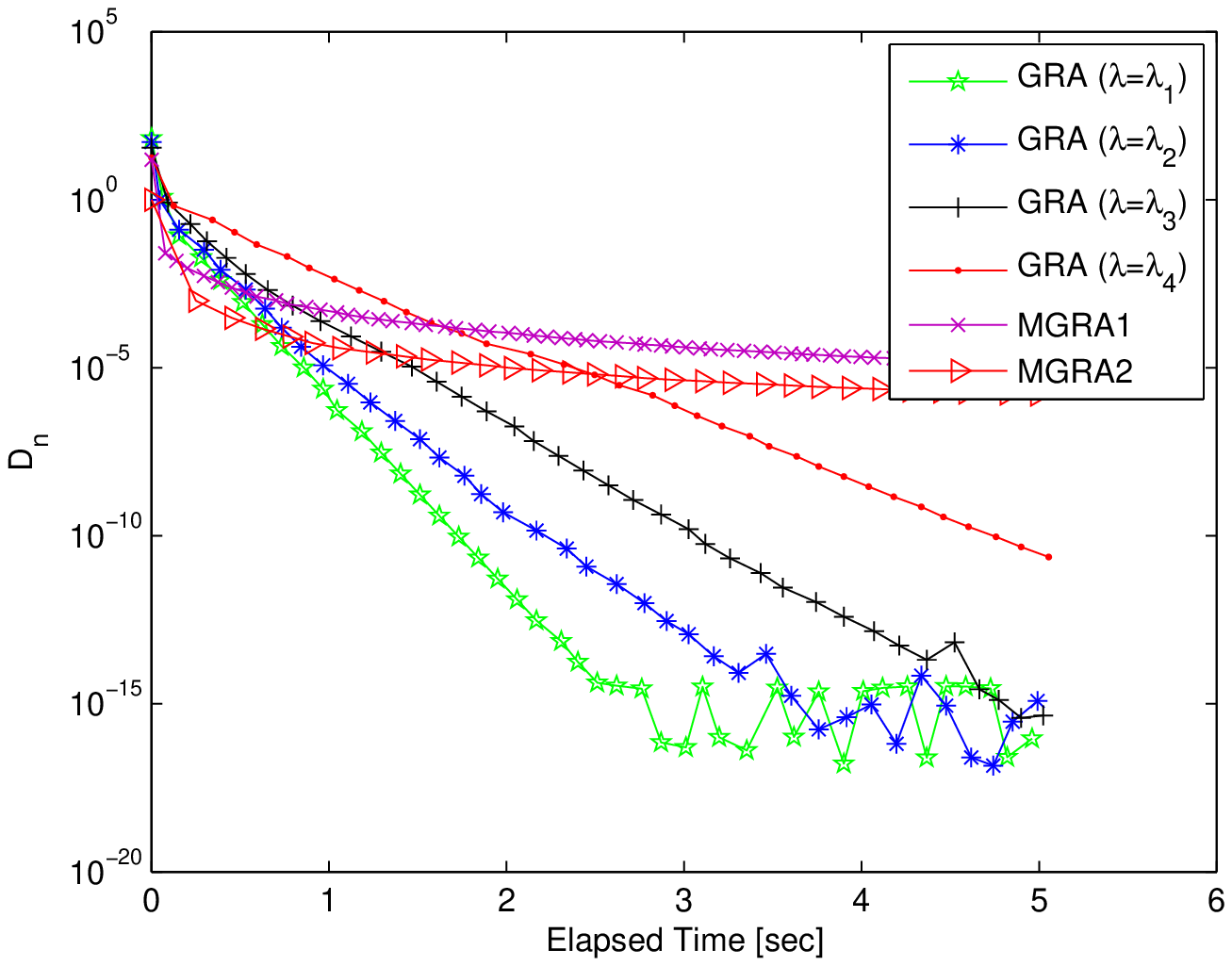}
\caption{$D_n$ and Time (in $\Re^{100}$)}\label{fig2}
\end{minipage}
\end{figure}

\begin{figure}[!ht]
\begin{minipage}[b]{0.45\textwidth}
\centering
\includegraphics[height=5cm,width=6cm]{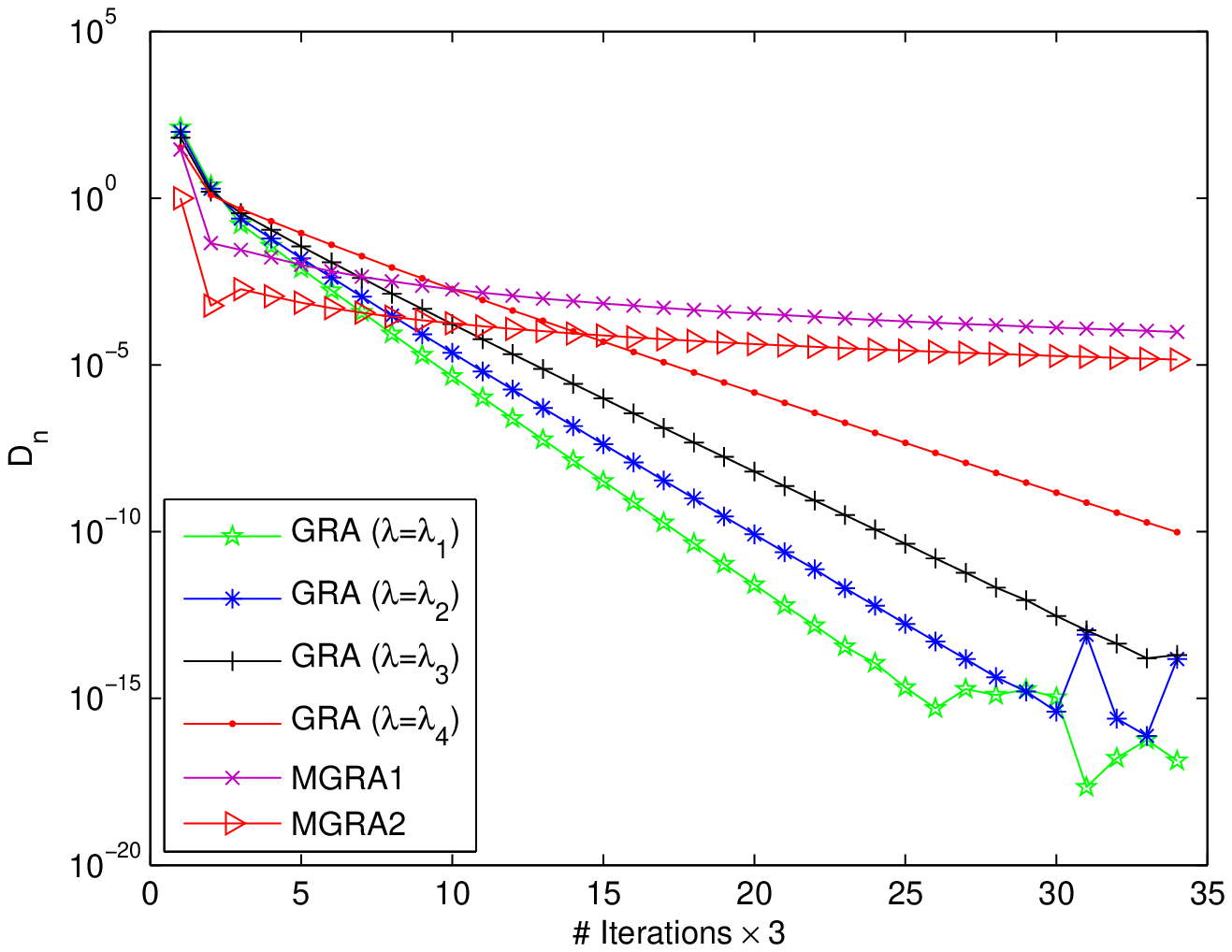}
\caption{$D_n$ and \# Iterations (in $\Re^{200}$)}\label{fig3}
\end{minipage}
\hfill
\begin{minipage}[b]{0.45\textwidth}
\centering
\includegraphics[height=5cm,width=6cm]{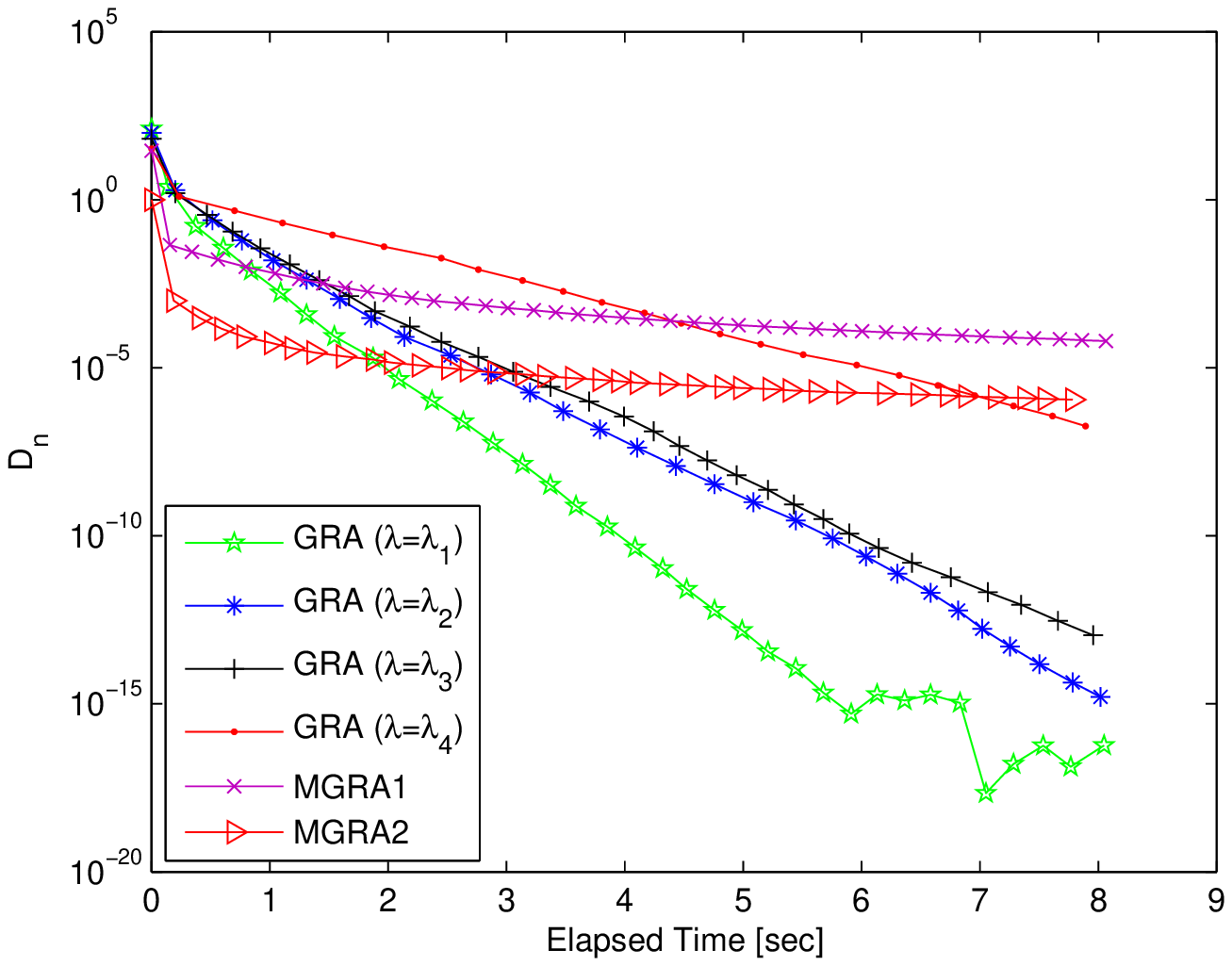}
\caption{$D_n$ and Time (in $\Re^{200}$)}\label{fig4}
\end{minipage}
\end{figure}

\begin{figure}[!ht]
\begin{minipage}[b]{0.45\textwidth}
\centering
\includegraphics[height=5cm,width=6cm]{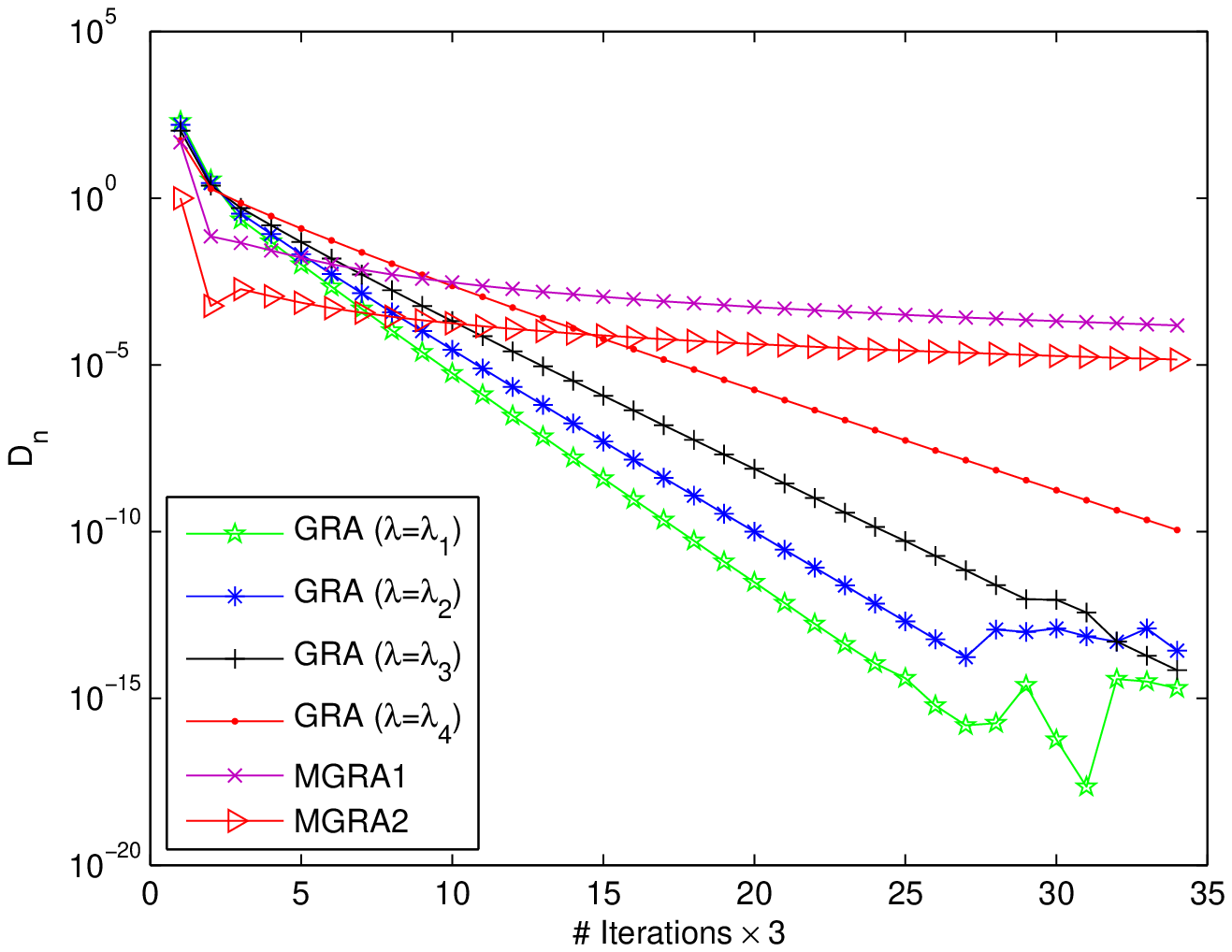}
\caption{$D_n$ and \# Iterations (in $\Re^{300}$)}\label{fig5}
\end{minipage}
\hfill
\begin{minipage}[b]{0.45\textwidth}
\centering
\includegraphics[height=5cm,width=6cm]{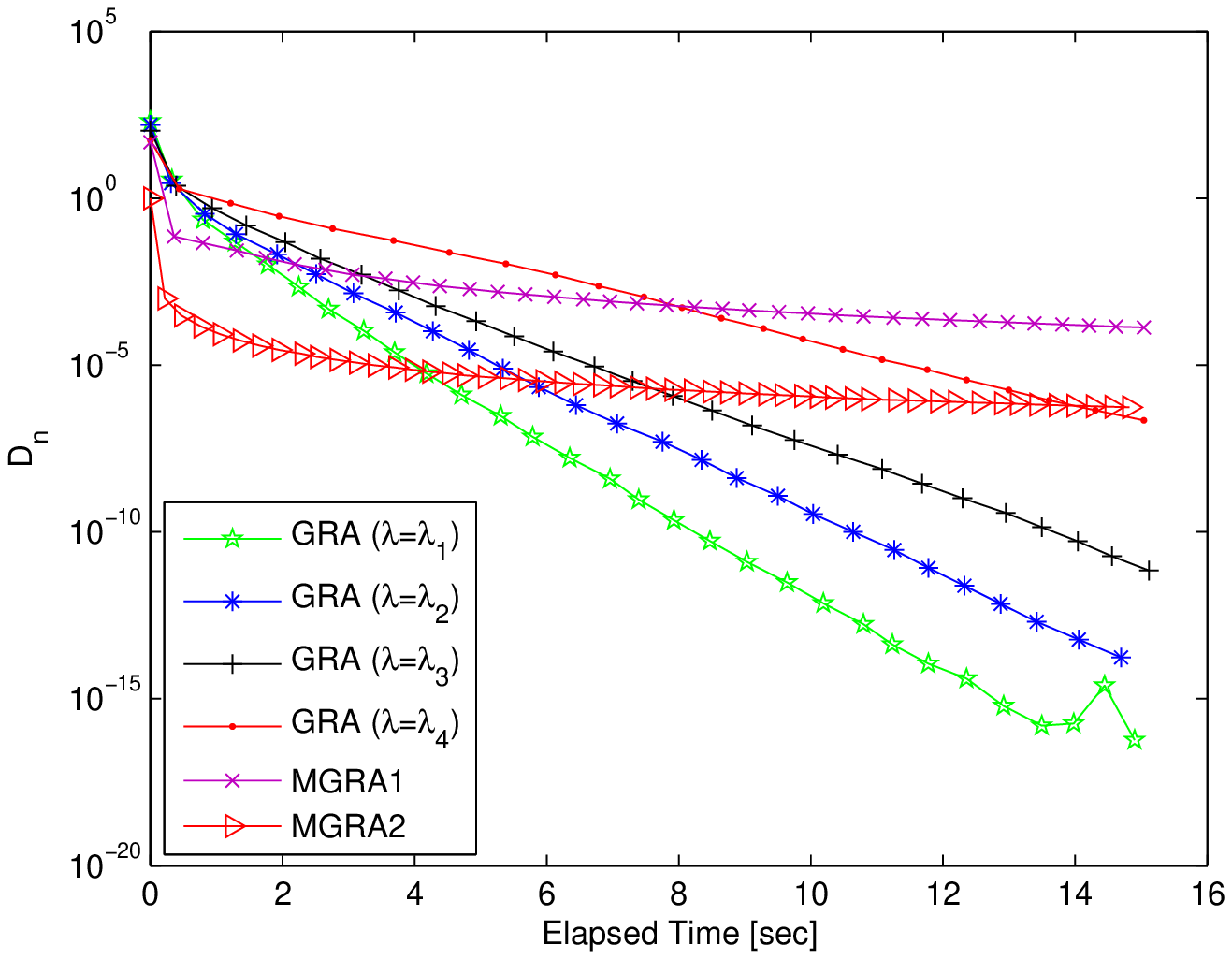}
\caption{$D_n$ and Time (in $\Re^{300}$)}\label{fig6}
\end{minipage}
\end{figure}
 
 \section*{Acknowledgements}
We would like to thank \textbf{Dr. Yura Malitsky} and \textbf{Dr. Nguyen The Vinh} for sending us the papers \cite{M2018,V2018} and for many valuable comments.

\end{document}